\documentclass[preprint,review,10pt]{amsart}
\usepackage{amsmath,amssymb,amsfonts,xspace}
\newtheorem{theorem}{Theorem}[section]

\theoremstyle{definition}
\newtheorem{definition}[theorem]{Definition}
\newtheorem{example}[theorem]{Example}

\newtheorem{corollary}[theorem]{Corollary}
\newtheorem{lem}[theorem]{Lemma}

\theoremstyle{remark}

\numberwithin{equation}{section}

\begin{document}

\title{ Krasner $(m,n)$-hyperring of fractions  }

\author{M. Anbarloei}
\address{Department of Mathematics, Faculty of Sciences,
Imam Khomeini International University, Qazvin, Iran.
}

\email{m.anbarloei@sci.ikiu.ac.ir }

%\author{}
%\address{}
%\email{}

%    General info
\subjclass[2010]{ 20N20, 19Y99, 20N15}

%\date{September  , 2013.}

\keywords{ $n$-ary prime hyperideal, $n$-ary multiplicative subset, Krasner $(m,n)$-hyperring.}
%------------------------------------------------------------------------------
%%%%%%%%%%%%%%%%%%%%%%%%%%%%%%%%%%%%%%%%%%%%%%%%%%%%%%%%%%%%%%%%%%%%%%%%%%%%%%%%%%%%%%%%%%%%%%%%%%%%%%
%%%%%%%%%%%%%%%%%%%%%%%%%%%%%%%%%%%%%%%%%%%%%%%%%%%%%%%%%%%%%%%%%%%%%%%%%%%%%%%%%%%%%%%%%%%%%%%%%%%%%%%%%%%%%%%%%%%%%%%%%%%%%%%%%%%%%%%%%
\begin{abstract}
 The formation of rings of fractions and the associated process of localization are  the most important technical tools in commutative algebra. Krasner $(m,n)$-hyperrings are  a generalization of  $(m,n)$-rings. Let $R$ be a commutative Krasner $(m,n)$-hyperring. The aim of this research work is to  introduce  the concept of  hyperring
of fractions generated by $R$ and then investigate the basic properties
such hyperrings. 
\end{abstract}
%%%%%%%%%%%%%%%%%%%%%%%%%%%%%%%%%%%%%%%%%%%%%%%%%%%%%%%%%%%%%%%%%%%%%%%%%%%%%%%%%
\maketitle
\section{Introduction}
%The importance  of   algebraic structures  in mathematics is significant. They have wide ranging  in both  applied and  pure sciences. They can be used in  computer sciences, coding theory, information sciences, theorical physics (see \cite{c1},\cite{c4},\cite{c7}, \cite{c3},\cite{c2},\cite{c5} and \cite{c6}). Hyperstructures are an extention of ordinary algebraic structures. Introduction of the structures  was first done by  Marty at the $8^{th}$ Congress of Scandinavian Mathematicians in 1934  \cite{sorc1}. Marty introduced hypergroups as a generalization of groups.  Later on, various authors studied it such as by Mittas \cite{q1111},\cite{q1110} and Corsini \cite{sorc2},\cite{sorc3}.\\ In 1904, Kasner introduced the notion of n-ary algebras in a lecture in an annual meeting \cite{mim}. It is worth noting that ِDorente wrote the first paper on the theory of n-ary groups in 1928 \cite{q1112}.
The notion of Krasner hyperrings was introduced by Krasner for the
first time in \cite{miim}. Also, we can see some properties on Krasner hyperrings in \cite{dav3} and \cite{dav4}. In \cite{dav}, Davvaz and Vougiouklis defined the notion of $n$-ary hypergroups which is a generalization of hypergroups in the sense of Marty.  The concept of $(m,n)$-ary hyperrings was introduced in \cite{dav2}. Davvaz and et. al. introduced Krasner $(m, n)$- hyperrings as a generalization of $(m, n)$-rings and studied  some results in this context in \cite{d1}. We can see some important hyperideals of the Krasner $(m,n)$-hyperrings in \cite{amer} and \cite{dav5}. Also, Ostadhadi and Davvaz studied the isomorphism theorems of ring theory and Krasner hyperring theory which are derived in the context of Krasner (m, n)-hyperrings in \cite{dav1}.
Ameri and Norouzi introduced in \cite{amer} the notions of $n$-ary prime and $n$-ary primary hyperideals in a Krasner $(m,n)$-hyperring and proved some results in this respect. The notion of $n$-ary 2-absorbing hyperideals  in a Krasner $(m,n)$-hyperring as a generalization of the $n$-ary prime hyperideals  was introduced  in \cite{mah}. To unify the concepts of the prime and
primary hyperideals under one frame, the notion of $\delta$-primary hyperideals was defined in Krasner $(m,n)$-hyperrings \cite{mah2}. The formation of rings of fractions and the associated process of localization are  the most important technical tools in commutative algebra. They correspond in the algebra-geometric picture to concentratining attention on an open set or near a point, and the importance of these notions should be self-evident. Procesi and Rota in \cite{88888} have studied ring of fractions in Krasner hyperrings.

In this paper, we aim to define the notion of a hyperring of fractions of Krasner $(m,n)$-hyperrings  and provide several properties of them.
The paper is organized as follows. In section 2, we have given some basic definitions and results of n-ary hyperstructures which we need to develop our paper.
In section 3, we have constructed the Krasner $(m,n)$-hyperring of fractions. In section 4, we have studied the   hyperideals of 
Krasner $(m,n)$-hyperring of fractions. In section 5, we have investigated construction of  qutient Krasner $(m,n)$-hyperring of fractions. 
 %%%%%%%%%%%%%%%%%%%%%% 
\section{Preliminaries}
In this section we recall some definitions and results about $n$-ary hyperstructures which we need to develop our paper.\\
A mapping $f : H^n \longrightarrow P^*(H)$
is called an $n$-ary hyperoperation, where $P^*(H)$ is the
set of all the non-empty subsets of $H$. An algebraic system $(H, f)$, where $f$ is an $n$-ary hyperoperation defined on $H$, is called an $n$-ary hypergroupoid. \\
We shall use the following abbreviated notation:\\
The sequence $x_i, x_{i+1},..., x_j$ 
will be denoted by $x^j_i$. For $j < i$, $x^j_i$ is the empty symbol. In this convention

$f(x_1,..., x_i, y_{i+1},..., y_j, z_{j+1},..., z_n)$\\
will be written as $f(x^i_1, y^j_{i+1}, z^n_{j+1})$. In the case when $y_{i+1} =... = y_j = y$ the last expression will be written in the form $f(x_1^i, y^{(j-i)}, z^n_{j+1})$. 
For non-empty subsets $A_1,..., A_n$ of $H$ we define

$f(A^n_1) = f(A_1,..., A_n) = \bigcup \{f(x^n_1) \ \vert \ x_i \in A_i, i = 1,..., n \}$.\\
An $n$-ary hyperoperation $f$ is called associative if

$f(x^{i-1}_1, f(x_ i ^{n+i-1}), x^{2n-1}_{n+i}) = f(x^{j-1}_1, f(x_j^{n+j-1}), x_{n+j}^{2n-1}),$ \\
hold for every $1 \leq i < j \leq n$ and all $x_1, x_2,..., x_{2n-1} \in H$. An $n$-ary hypergroupoid with the
associative $n$-ary hyperoperation is called an $n$-ary semihypergroup. 

An $n$-ary hypergroupoid $(H, f)$ in which the equation $b \in f(a_1^{i-1}, x_i, a_{ i+1}^n)$ has a solution $x_i \in H$
for every $a_1^{i-1}, a_{ i+1}^n,b \in H$ and $1 \leq i \leq n$, is called an $n$-ary quasihypergroup, when $(H, f)$ is an $n$-ary
semihypergroup, $(H, f)$ is called an $n$-ary hypergroup. 

An $n$-ary hypergroupoid $(H, f)$ is commutative if for all $ \sigma \in \mathbb{S}_n$, the group of all permutations of $\{1, 2, 3,..., n\}$, and for every $a_1^n \in H$ we have $f(a_1,..., a_n) = f(a_{\sigma(1)},..., a_{\sigma(n)})$.
If an $a_1^n \in H$ we denote $a_{\sigma(1)}^{\sigma(n)}$ as the $(a_{\sigma(1)},..., a_{\sigma(n)})$. We assume throughout this paper that all Krasner $(m,n)$-hyperrings are commutative.

If $f$ is an $n$-ary hyperoperation and $t = l(n- 1) + 1$, then $t$-ary hyperoperation $f_{(l)}$ is given by

$f_{(l)}(x_1^{l(n-1)+1}) = f(f(..., f(f(x^n _1), x_{n+1}^{2n -1}),...), x_{(l-1)(n-1)+1}^{l(n-1)+1})$. 
\begin{definition}
\cite{d1} Let $(H, f)$ be an $n$-ary hypergroup and $B$ be a non-empty subset of $H$. $B$ is called
an $n$-ary subhypergroup of $(H, f)$, if $f(x^n _1) \subseteq B$ for $x^n_ 1 \in B$, and the equation $b \in f(b^{i-1}_1, x_i, b^n _{i+1})$ has a solution $x_i \in B$ for every $b^{i-1}_1, b^n _{i+1}, b \in B$ and $1 \leq i \leq n$.
An element $e \in H$ is called a scalar neutral element if $x = f(e^{(i-1)}, x, e^{(n-i)})$, for every $1 \leq i \leq n$ and
for every $x \in H$. 

An element $0$ of an $n$-ary semihypergroup $(H, g)$ is called a zero element if for every $x^n_2 \in H$ we have
$g(0, x^n _2) = g(x_2, 0, x^n_ 3) = ... = g(x^n _2, 0) = 0$.
If $0$ and $0^ \prime $are two zero elements, then $0 = g(0^ \prime , 0^{(n-1)}) = 0 ^ \prime$ and so the zero element is unique. 
\end{definition}
\begin{definition}
\cite{l1} Let $(H, f)$ be a $n$-ary hypergroup. $(H, f)$ is called a canonical $n$-ary
hypergroup if\\
(1) there exists a unique $e \in H$, such that  $f(x, e^{(n-1)}) = x$ for every $x \in H$;\\
(2) for all $x \in H$ there exists a unique $x^{-1} \in H$, such that $e \in f(x, x^{-1}, e^{(n-2)})$;\\
(3) if $x \in f(x^n _1)$, then  $x_i \in f(x, x^{-1},..., x^{-1}_{ i-1}, x^{-1}_ {i+1},..., x^{-1}_ n)$ for all $i$ .

We say that e is the scalar identity of $(H, f)$ and $x^{-1}$ is the inverse of $x$. Notice that  $e^{-1}=e$ 
\end{definition}
\begin{definition}
\cite{d1} A Krasner $(m, n)$-hyperring is an algebraic hyperstructure $(R, f, g)$ which
satisfies the following axioms:\\
(1) $(R, f$) is a canonical $m$-ary hypergroup;\\
(2) $(R, g)$ is an $n$-ary semigroup;\\
(3) the $n$-ary operation $g$ is distributive with respect to the $m$-ary hyperoperation $f$ , i.e., 
\[g(a^{i-1}_1, f(x^m _1 ), a^n _{i+1}) = f(g(a^{i-1}_1, x_1, a^n_{ i+1}),..., g(a^{i-1}_1, x_m, a^n_{ i+1}))\] for every $a^{i-1}_1 , a^n_{ i+1}, x^m_ 1 \in R$ and $1 \leq i \leq n$;\\
(4) $0$ is a zero element (absorbing element) of the $n$-ary operation $g$, i.e.,   
\[g(0, x^n _2) = g(x_2, 0, x^n _3) = ... = g(x^n_ 2, 0) = 0\]
for every $x^n_ 2 \in R$.\\We denote the Krasner (m,n)-hyperring $(R,f,g)$ simply by $R$. We say that $R$ is with scalar identity if there exists an element $1$ such that $x = g(x, 1^{( n-1)})$
for all $x \in R$. In this paper, we assume that $R$ is with scalar identity.

\end{definition}
A non-empty subset $S$ of $R$ is said to be a subhyperring of $R$ if $(S, f, g)$ is a Krasner $(m, n)$-hyperring. Let
$I$ be a non-empty subset of $R$, we say that $I$ is a hyperideal of $R$ if $(I, f)$ is an $m$-ary subhypergroup
of $(R, f)$ and $g(x^{i-1}_1, I, x_{i+1}^n) \subseteq I$, for every $x^n _1 \in R$ and $1 \leq i \leq n$.\\

\begin{definition}
\cite{amer} A proper hyperideal $I$ of a Krasner $(m, n)$-hyperring $R$ is said to be an $n$-ary prime hyperideal if for hyperideals $I_1,..., I_n$ of $R$, $g(I_1^ n) \subseteq I$ implies that $I_1 \subseteq I$ or $I_2 \subseteq I$ or ...or $I_n \subseteq I$.
\end{definition}
\begin{lem} 
A proper hyperideal   $I$ of a Krasner $(m, n)$-hyperring $R$ is an $n$-ary prime hyperideal if for all $x^n_ 1 \in R$,
$g(x^n_ 1) \in I $ implies that  $x_1 \in I \ or \ ... \ or \ x_n \in I$. (Lemma 4.5 in \cite{amer})
\end{lem}
\begin{definition}
\cite{amer}  Let $R$ be a Krasner $(m, n)$-hyperring. A non-empty subset $S$ of $R$ is called $n$-ary
multiplicative, if $g(s_1^n) \in S$  for $s_1,..., s_n \in S$. 
\end{definition}
In this paper, we assume that $1 \in S$.
\begin{definition}
\cite{amer} A Krasner $(m, n)$-hyperring $R$ is said to be an $n$-ary hyperintegral domain, if $R$ is a
commutative Krasner $(m, n)$-hyperring and $g(x_1^n) = 0$ implies that $x_1 = 0$ or $x_2 = 0$ or ... or $x_n = 0$
for all $x_1^n$.
\end{definition}
\begin{definition}
 \cite{amer} Let $R$ be a Krasner $(m, n)$-hyperring. An element $x \in  R$ is said to be invertible if there exists $y \in  R$ with  $1= g(x, y, 1^{( n-2)})$. 
\end{definition}
\begin{definition} \cite{d1} 
Let $(R_1, f_1, g_1)$ and $(R_2, f_2, g_2)$ be two Krasner $(m, n)$-hyperrings. A mapping
$\phi : R_1 \longrightarrow R_2$ is called a homomorphism if for all $x^m _1 \in R_1$ and $y^n_ 1 \in R_1$ we have

$\phi(f_1(x_1,..., x_m)) = f_2(\phi(x_1),...,\phi(x_m))$

$\phi(g_1(y_1,..., y_n)) = g_2(\phi(y_1),...,\phi(y_n)). $
\end{definition}

%If $Q$ is an $n$-ary primary hyperideal in a Krasner $(m, n)$-hyperring $(R, f, g)$ with the scalar identity $1_R$, then ${\sqrt Q}^{(m,n)}$ is $n$-ary prime. (Theorem 4.28 in \cite{sorc1})
%\begin{definition} (\cite{sorc1})
%Let $S$ be a hyperideal of a Krasner $(m, n)$-hyperring $(R, f, g)$. Then the set

%$R/S = \{f(x^{i-1}_1, S, x^m_{i+1}) \ \vert \ x^{i-1}_1,x^m_{i+1} \in R \}$\\
%endowed with m-ary hyperoperation $f$ which for all $x_{11}^{1m},...,x_{m1}^{mm} \in R$

%$f(f(x_{11}^{1 (i-1)}, S, x^{1m}_ {1(i+1)}),..., f(x_{m1}^{ m(i-1)}, S, x^{mm}_ {m(i+1)}))$ 

%$= f (f(x^{m1}_{11}),..., f(x^{m(i-1)}_{1(i-1)}), S, f(x^{m(i+1)}_{1(i+1)} ),..., f(x^{mm}_ {1m}))$\\
%and with $n$-ary hyperoperation g which for all $x_{11}^{1m},...,x_{n1}^{nm} \in R$

%$g(f(x_{11}^{1 (i-1)}, S, x^{1m}_ {1(i+1)}),..., f(x_{n1}^{ n(i-1)}, S, x^{nm}_ {n(i+1)}))$ 

%$= f (g(x^{n1}_{11}),..., g(x^{n(i-1)}_{1(i-1)}), S, g(x^{n(i+1)}_{1(i+1)} ),..., f(x^{nm}_ {1m}))$\\
%construct a Krasner $(m, n)$-hyperring, and $(R/S, f, g)$ is called the quotient Krasner $(m, n)$-hyperring of $R$ by $S$. 
%\end{definition}
%begin{definition} (\cite{d1})
%Let $(R_1, f_1, g_1)$ and $(R_2, f_2, g_2)$ be two Krasner $(m, n)$-hyperrings. A mapping
%$\phi : R_1 \longrightarrow R_2$ is called a homomorphism if for all $x^m _1 \in R_1$ and $y^n_ 1 \in R_1$ we have

%$\phi(f_1(x_1,..., x_m)) = f_2(\phi(x_1),...,\phi(x_m))$

%$\phi(g_1(y_1,..., y_n)) = g_2(\phi(y_1),...,\phi(y_n)). $
%\end{definition}
%%%%%%%%%%%%%%%%%%%%%%%%%%%%%%%%%%%%%%%%%%%%%%%%%%%%%%%%%%%%%%%%%%%%%%%%%%%%

\section{ Krasner $(m,n)$-hypering of fractions }
 Let $R$ be any Krasner $(m,n)$-hyperring and
let $S$ be an $n$-ary multiplicative subset of $R$ such that  $1 \in  S$. We shall construct the Krasner $(m,n)$-hyperring of fractions
$S^{-1}R$. We define a relation $\sim$ on $R \times  S$ by $(r,s) \sim (r^\prime,s^\prime)$ if and
only if there exists some $s \in S$ such that 

$0 \in g(s,f(g(r,s^\prime,1^{(n-2)}),g(r^\prime,s,1^{(n-2)})
,0^{(m-2)}),1^{(n-2)})$.
\begin{theorem}
The relation $\sim$ is an equivalence relation on $R \times S$.
\end{theorem}
\begin{proof}
Clearly, $\sim$ is reflexive and symmetric. Suppose that $(r_1, s_1) \sim (r_2, s_2)$ and $(r_2, s_2) \sim (r_3, s_3)$. Then there exist $ s \in S$ such that

$0 \in g(s,f(g(r_1,s_2,1^{(n-2)}),-g(r_2,s_1,1^{(n-2)}),0^{(m-2)}),1^{(n-2)})$\\
and

$0 \in g(s^\prime,f(g(r_2,s_3,1^{(n-2)}),-g(r_3,s_2,1^{(n-2)}),0^{(m-2)}),1^{(n-2)})$.\\
Since 

$0 \in g(s,f(g(r_1,s_2,1^{(n-2)}),-g(r_2,s_1,1^{(n-2)}),0^{(m-2)}),1^{(n-2)})$

$\hspace{0.3cm}=f(g(s,r_1,s_2,1^{(n-3)}),-g(s,r_2,s_1,1^{(n-2)}),0^{(m-2)})$,\\  we get $g(s,r_2,s_1,1^{(n-2)}) \in f(g(s,r_1,s_2,1^{(n-3)}),0^{(m-1)})$.\\
 Thus we have\\
$0=g(g(s,s_1,1^{(n-2)}),0^{(n-1)})$

$ \in g(g(s,s_1,1^{(n-2)}),g(s^\prime,f(g(r_2,s_3,1^{(n-2)}),-g(r_3,s_2,1^{(n-2)}),0^{(m-2)}),1^{(n-2)}),$

$\hspace{0.2cm}1^{(n-2)})$

$=g(g(s,s_1,1^{(n-2)}),f(g(s^\prime,r_2,s_3,1^{(n-3)}),
-g(s^\prime,r_3,s_2,1^{(n-3)}),0^{(m-2)}),1^{(n-2)})$

$=f(g(s,s_1,s^\prime,r_2,s_3,1^{(n-5)}),
-g(s,s_1,s^\prime,r_3,s_2,1^{(n-5)}),0^{(m-2)})$

$=f(g(s^\prime,g(s,r_2,s_1,1^{(n-3)}),s_3),-g(s,s_1,s^\prime,r_3,s_2,1^{(n-3)}),0^{(m-3)})$

$\subseteq f(g(s^\prime,f(g(s,r_1,s_2,1^{(n-3)}),0^{(m-1)}),s_3,1^{(n-3)}),-g(s,s_1,s^\prime,r_3,s_2,1^{(n-3)}),$

$\hspace{0.2cm} 0^{(m-3)})$

$=f(g(s,s^\prime,s_2,r_1,s_3),-g(s,s^\prime,s_2,r_3,s_1),0^{(m-2)})$

$=g(g(s,s^\prime,s_2,1^{(n-3)}),f(g(r_1,s_3,1^{(n-2)})
,-g(r_3,s_1,1^{(n-2)}),0^{(m-2)}),1^{(n-2)}).$\\
Since $g(s,s^\prime,s_2,1^{(n-3)}) \in S$, then $(r_1,s_1) \sim (r_3,s_3)$.
Consequently,  $\sim$ is transitive.
\end{proof}
We denote the equivalence class of $(a, s)$ with $\frac{r}{s}$ and let $S^{-1}R$ denote the set
of all equivalence classes. 
We endow the set $S^{-1}R$ with a Krasner $(m,n)$-hyperring structure, by defining the $m$-ary hyperoperation $F$ and the $n$-ary operation $G$
 as follows: \\
 $F(\frac{r_1}{s_1},...,\frac{r_m}{s_m})=\frac{f(g(r_1,s_2^m,1^{(n-m)}),g(s_1,r_2,s_3^m,1^{(n-m)}),...,g(s_1^{m-1},r_m,1^{(n-m)})}{g(s_1^m,1^{(n-m)})}$\\
 $=\{\frac{r}{s}\ \vert \ r \in f(g(r_1,s_2^m,1^{(n-m)}),g(s_1,r_2,s_3^m,1^{(n-m)}),...,g(s_1^{m-1},r_m,1^{(n-m)}),$
 
 $s=g(s_1^m)\}$\\
 $G(\frac{r_1}{s_1},...,\frac{r_n}{s_n})=\frac{g(r_1^n)}{g(s_1^n)}.$\\
 We need to show that $F$ and $G$ are well defined. If $\frac{r_1}{s_1}=\frac{r_1^\prime}{s_1^\prime}$, $\frac{r_2}{s_2}=\frac{r_2^\prime}{s_2^\prime},...,\frac{r_m}{s_m}=\frac{r_m^\prime}{s_m^\prime}$, then there exist $t_1,...,t_m \in S$ such that 
 %\begin{align*}
 %g(t_1,r_1,s_1^\prime,1^{(n-3)})&=g(t_1,r_1^\prime,s_1,1^{(n-3)}) & (1)\\
 % g(t_2,r_2,s_2^\prime,1^{(n-3)})&=g(t_2,r_2^\prime,s_2,1^{(n-3)}) &(2)\\
  %\vdots&\\
  % g(t_m,r_m,s_m^\prime,1^{(n-3)})&=g(t_m,r_m^\prime,s_m^\prime,1^{(n-3)}) &(m)
% \end{align}
\begin{align*}
0&\in g(t_1,f(g(r_1,{s^\prime}_1,1^{(n-2)}),-g({r^\prime}_1,s_1,1^{(n-2)}),0^{m-2}),1^{(n-2)}) &(1)\\
0&\in g(t_2,f(g(r_2,{s^\prime}_2,1^{(n-2)}),-g({r^\prime}_2,s_2,1^{(n-2)}),0^{m-2}),1^{(n-2)}) &(2)\\
&\vdots\\
0&\in g(t_m,f(g(r_m,{s^\prime}_m,1^{(n-2)}),-g({r^\prime}_m,s_m,1^{(n-2)}),0^{m-2}),1^{(n-2)}). &(m)
\end{align*}
 $g$-producting  (1) by
 
  $g(g(t_2^m,1^{(n-m+1)}),g(1^{(n-m+1)},s_2^m),g(1^{(n-m+1)},{s^\prime}_2^m),1^{(n-3)})$, \\(2) by
  
   $g(g(t_1,1^{(n-m+1)},t_3^m),(s_1,1^{(n-m+1)},s_3^m),g({s^\prime} _1,1^{(n-m+1)},{s^\prime} _3^m),1^{(n-3)})$
   
   $\vdots$ \\
  (m) by 
   
   $g(g(t_1^{m-1},1^{(n-m+1)}),g(s_1^{m-1},1^{(n-m+1)}),g({s^\prime}_1^{m-1},1^{(n-m+1)}),1^{(n-3)})$. \\Thus we get
   
   %$g(g(t_1^m,1^{(n-m)}),g({s^\prime}_1^m,1^{(n-m)}),g(r_1,s_2^m,1^{(n-m)}),1^{(n-3)})$
   
%$\hspace{0.5cm}   =
   %g(g(t_1^m,1^{(n-m)}),g(s_1^m,1^{(n-m)}),g(r_1^\prime,{s^\prime}_2^m,1^{(n-m)}),1^{(n-3)})$

 %$g(g(t_1^m,1^{(n-m)}),g({s^\prime}_1^m,1^{(n-m)}),g(r_1,s_1,s_3^m,1^{(n-m)}),1^{(n-3)})$
   
%$\hspace{0.5cm}   =
   %g(g(t_1^m,1^{(n-m)}),g(s_1^m,1^{(n-m)}),g(r_2^\prime.{s^\prime}_1,{s^\prime}_3^m,1^{(n-m)}),1^{(n-3)})$
   
   %$\vdots$
   
   %$g(g(t_1^m,1^{(n-m)}),g({s^\prime}_1^m,1^{(n-m)}),g(r_m,s_1^{m-1},1^{(n-m)}),1^{(n-3)})$
   
%$\hspace{0.5cm}   =
   %g(g(t_1^m,1^{(n-m)}),g(s_1^m,1^{(n-m)}),g(r_m^\prime,{s^\prime}_1^{m-1},1^{(n-m)}),1^{(n-3)})$\\
 
   $0 \in g(g(t_1^m,1^{(n-m)}),f(g(g({s^\prime}_1^m,1^{(n-m)}),g(r_1,s_2^m,1^{(n-m)}),1^{(n-2)}),$
   
   $\hspace{0.5cm}-g(
   (g(s_1^m,1^{(n-m)}),g(r_1^\prime,{s^\prime}_2^m,1^{(n-m)}),1^{(n-2)}),0^{(m-2)}),1^{(n-2)})$
  
   $0\in g(g(t_1^m,1^{(n-m)}),f(g(g({s^\prime}_1^m,1^{(n-m)}),g(r_1,s_1,s_3^m,1^{(n-m)}),1^{(n-2)}),$
   
$\hspace{0.5cm}   -g(
  g(s_1^m,1^{(n-m)}),g(r_2^\prime.{s^\prime}_1,{s^\prime}_3^m,1^{(n-m)}),1^{(n-2)}),0^{(m-2)}),1^{(n-2)})$
   
  $\vdots$
   
  $0 \in g(g(t_1^m,1^{(n-m)}),f(g({s^\prime}_1^m,1^{(n-m)}),g(r_m,s_1^{m-1},1^{(n-m)}),1^{(n-2)}),$
   
$\hspace{0.5cm}   -
  g(g(s_1^m,1^{(n-m)}),g(r_m^\prime,{s^\prime}_1^{m-1},1^{(n-m)}),1^{(n-2)}),0^{(m-2)}),1^{(n-2)})$.\\
   Now, we have

   %$f(g(g(t_1^m,1^{(n-m)}),g({s^\prime}_1^m,1^{(n-m)}),g(r_1,s_2^m,1^{(n-m)}),1^{(n-3)}),$
   
   %$\hspace{0.3cm}g(g(t_1^m,1^{(n-m)}),g({s^\prime}_1^m,1^{(n-m)}),g(r_1,s_1,s_3^m,1^{(n-m)}),1^{(n-3)}),$
   
   %$\hspace{0.3cm} \cdots,$
   
   %$\hspace{0.3cm} g(g(t_1^m,1^{(n-m)}),g({s^\prime}_1^m,1^{(n-m)}),g(r_m,s_1^{m-1},1^{(n-m)}),1^{(n-3)}))$
   
  % $=f(g(g(t_1^m,1^{(n-m)}),g(s_1^m,1^{(n-m)}),g(r_1^\prime,{s^\prime}_2^m,1^{(n-m+1)}),1^{(n-3)}),$
   
   %$\hspace{0.3cm}g(g(t_1^m,1^{(n-m)}),g(s_1^m,1^{(n-m)}),g(r_2^\prime.{s^\prime}_1,{s^\prime}_3^m,1^{(n-m)}),1^{(n-3)}),$
   
   %$\hspace{0.3cm} \cdots,$
   
  % \hspace{0.3cm}g(g(t_1^m,1^{(n-m)}),g(s_1^m,1^{(n-m)}),g(r_m^\prime,{s^\prime}_1^{m-1},1^{(n-m)}),1^{(n-3)}))$.\\
   
 $0\in f(f(g(g(t_1^m,1^{(n-m)}),g({s^\prime}_1^m,1^{(n-m)}),g(r_1,s_2^m,1^{(n-m)}),1^{(n-3)}),$
   
 $\hspace{0.3cm}g(g(t_1^m,1^{(n-m)}),g({s^\prime}_1^m,1^{(n-m)}),g(r_1,s_1,s_3^m,1^{(n-m)}),1^{(n-3)}),$
   
  $\hspace{0.3cm} \cdots,$
   
  $\hspace{0.3cm} g(g(t_1^m,1^{(n-m)}),g({s^\prime}_1^m,1^{(n-m)}),g(r_m,s_1^{m-1},1^{(n-m)}),1^{(n-3)})),$
   
   $-f(g(g(t_1^m,1^{(n-m)}),g(s_1^m,1^{(n-m)}),g(r_1^\prime,{s^\prime}_2^m,1^{(n-m+1)}),1^{(n-3)}),$
   
   $\hspace{0.3cm}g(g(t_1^m,1^{(n-m)}),g(s_1^m,1^{(n-m)}),g(r_2^\prime.{s^\prime}_1,{s^\prime}_3^m,1^{(n-m)}),1^{(n-3)}),$
   
   $\hspace{0.3cm} \cdots,$
   
   $\hspace{0.3cm}g(g(t_1^m,1^{(n-m)}),g(s_1^m,1^{(n-m)}),g(r_m^\prime,{s^\prime}_1^{m-1},1^{(n-m)}),1^{(n-3)}),0^{(m-2)})$.\\
   
  We put
   $ t=g(t_1^m,1^{(n-m)}), s=g(s_1^m,1^{(n-m)}))$ and $s^\prime=g({s^\prime}_1^m,1^{(n-m)})$.\\
  Therefore we have
  %$g(t,g(s^\prime,f(g(r_1,s_2^m,1^{(n-m)}),g(r_2^\prime.{s^\prime}_1,{s^\prime}_3^m,1^{(n-m)}),\cdots,
%g(r_m,s_1^{m-1},1^{(n-m)}),1^{(n-2)})$\\ 
%$= g(t,g(s,f(g(r_1^\prime,{s^\prime}_2^m,1^{(n-m)}),
%g(r_2^\prime.{s^\prime}_1,{s^\prime}_3^m,1^{(n-m)}),
%\cdots,g(r_m^\prime,{s^\prime}_1^{m-1},1^{(n-m)}),1^{(n-2)})$. 

$0 \in f(g(t,g(s^\prime,f(g(r_1,s_2^m,1^{(n-m)}),g(r_2^\prime.{s^\prime}_1,{s^\prime}_3^m,1^{(n-m)}),\cdots,
g(r_m,s_1^{m-1},1^{(n-m)}),$

$\hspace{0.5cm}1^{(n-2)}),- g(t,g(s,f(g(r_1^\prime,{s^\prime}_2^m,1^{(n-m)}),
g(r_2^\prime.{s^\prime}_1,{s^\prime}_3^m,1^{(n-m)}),
\cdots,g(r_m^\prime,{s^\prime}_1^{m-1},$

$\hspace{0.5cm}1^{(n-m)}),1^{(n-2)},0^{(m-2)})$.
\\Thus $F(\frac{r_1}{s_1},...,\frac{r_m}{s_m})=F(\frac{r^\prime_1}{s^\prime_1},...,\frac{r^\prime_m}{s^\prime_m})$, i. e.,  $F$ is well defined.

Now, suppose that $\frac{r_1}{s_1}=\frac{r_1^\prime}{s_1^\prime}$, $\frac{r_2}{s_2}=\frac{r_2^\prime}{s_2^\prime},...,\frac{r_n}{s_n}=\frac{r_n^\prime}{s_n^\prime}$, then there exist $t_1,...,t_n \in S$ such that 
%\begin{align*}
% g(t_1,r_1,s_1^\prime,1^{(n-3)})&=g(t_1,r_1^\prime,s_1,1^{(n-3)}) \\
 % g(t_2,r_2,s_2^\prime,1^{(n-3)})&=g(t_2,r_2^\prime,s_2,1^{(n-3)}) \\
  %\vdots&\\
  % g(t_n,r_n,s_n^\prime,1^{(n-3)})&=g(t_n,r_n^\prime,s_n^\prime,1^{(n-3)}) 
% \end{align}

$0 \in g(t_1,f(g(r_1,s_1^\prime,1^{(n-2)}),-g(r_1^\prime,s_1,1^{(n-2)}),0^{(m-2)}),1^{(n-2)})$

 $0 \in g(t_2,f(g(r_2,s_2^\prime,1^{(n-2)}),-g(r_2^\prime,s_2,1^{(n-2)}),0^{(m-2)}),1^{(n-2)})$

  $\vdots$

   $0 \in g(t_n,f(g(r_n,s_n^\prime,1^{(n-2)}),-g(r_n^\prime,s_n^\prime,1^{(n-2)}),0^{(m-2)}),1^{(n-2)})$.\\
 Then we conclude that
 
 %$g(g(t_1,r_1,s_1^\prime,1^{(n-3)}),g(t_2,r_2,s_2^\prime,1^{(n-3)}),..., g(t_n,r_n,s_n^\prime,1^{(n-3)}),1^{(n-m)})$
 
 %$\hspace{0.3cm}=g(g(t_1,r_1^\prime,s_1,1^{(n-3)}),g(t_2,r_2^\prime,s_2,1^{(n-3)}),...,g(t_n,r_n^\prime,s_n^\prime,1^{(n-3)}),1^{(n-m)}).$\\
 
 $0 \in f(g(g(t_1,r_1,s_1^\prime,1^{(n-3)}),g(t_2,r_2,s_2^\prime,1^{(n-3)}),..., g(t_n,r_n,s_n^\prime,1^{(n-3)}),1^{(n-m)}),$
 
 $\hspace{0.3cm}-g(g(t_1,r_1^\prime,s_1,1^{(n-3)}),g(t_2,r_2^\prime,s_2,1^{(n-3)}),...,g(t_n,r_n^\prime,s_n^\prime,1^{(n-3)}),1^{(n-m)}),0^{(m-2)}).$\\
 It means 
 
 %$g(g(t_1^n),g(r_1^n),g({s^\prime}_1^n),1^{(n-3)})
 %=g(g(t_1^n),g({r^\prime}_1^n),g(s_1^n),1^{(n-3)})$\\
 
 $0 \in f(g(g(t_1^n),g(r_1^n),g({s^\prime}_1^n),1^{(n-3)}),
 -g(g(t_1^n),g({r^\prime}_1^n),g(s_1^n),1^{(n-3)}),0^{(m-2)}).$\\
   Put $t=g(t_1^n)$. We have

% $g(t,g(r_1^n),g({s^\prime}_1^n),1^{(n-3)})
 %=g(t,g({r^\prime}_1^n),g(s_1^n),1^{(n-3)})$.\\
 
 $0 \in f(g(t,g(r_1^n),g({s^\prime}_1^n),1^{(n-3)}),
 -g(t,g({r^\prime}_1^n),g(s_1^n),1^{(n-3)}),0^{(m-2)})$ \\and so
 
 $0 \in g(t,f(g(g(r_1^n),g({s^\prime}_1^n),1^{(n-2)}),
 -g(g(t,g({r^\prime}_1^n),g(s_1^n),1^{(n-2)}),0^{(m-2)}),1^{(n-2)})$.\\ 
 It implies that $\frac{g(r_1^n)}{g(s_1^n)}=\frac{g({r^\prime}_1^n)}{g({s^\prime}_1^n)}$ and so  $G(\frac{r_1}{s_1},...,\frac{r_n}{s_n})=G(\frac{r^\prime_1}{s^\prime_1},...,\frac{r^\prime_n}{s^\prime_n})$, i. e.,  $G$ is well defined. 
 
 \begin{lem} \label{11}
 Let $R$ be a Krasner $(m,n)$-hyperring and $S$ be an $n$-ary multiplicative subset of $R$ with $1 \in S$. Then:

 $1)$ For all $s \in S$, $\frac{0}{1}=\frac{0}{s}=0_{S^{-1}R}$.

 $2)$ $\frac{r}{s}=0_{S^{-1}R}$, for $r \in R, s \in S$ if and only if there exists $t \in S$ such that 
 
 $\hspace{0.5cm}g(t,r,1^{n-2})=0$.

 $3)$ For all $s \in S$, $\frac{s}{s}=\frac{1}{1}=1_{S^{-1}R} $.

 $4)$ $\frac{g(r,s^{(m-1)},1^{(n-m)})}{g(s^\prime,s^{(m-1)},1^{(n-m)})}=\frac{g(r,1^{(n-1)})}{g(s^\prime,1^{(n-1)})}$, for $r \in R$ and $s,s^\prime \in S$.
\end{lem}
\begin{proof}
$(1)$ Let $t \in S$. Then for all $s \in S$ we have

$  0=g(t,s,0,1^{(n-3)})$
 
 $\hspace{0.3cm}  = g(t,g(0,s,1^{(n-2)}),1^{(n-2)})$

 $\hspace{0.3cm} = g(t,f(g(0,s,1^{(n-2)}),0^{(m-1)}),1^{(n-2)})$

 $\hspace{0.3cm} = g(t,f(g(0,s,1^{(n-2)}),-g(1,0,1^{(n-2)}),0^{(m-2)}),1^{(n-2)}).$\\
 Then we conclude that $\frac{0}{1}=\frac{0}{s}=0_{S^{-1}R}$. Now, we show that $\frac{0}{1}=0_{S^{-1}R}$. Let $r \in R$ and $s \in S$. Then 
 
 $F(\frac{r}{s},\frac{0}{1}^{(m-1)})=
 \{\frac{u}{v} \ \vert \ u \in f(g(r,1^{(n-1)}), g(0,s,1^{(n-2)})^{(m-1)}) , s=g(s,1^{n-1}) \}$
 
 $\hspace{2.1cm}=
 \{\frac{u}{v} \ \vert \ u \in f(g(r,1^{(n-1)}), 0^{(m-1)}) , v=g(s,1^{n-1}) \}$

  $\hspace{2.1cm}=
 \{\frac{u}{v} \ \vert \ u \in f(r, 0^{(m-1)}) , s=g(s,1^{n-1}) \}$
 
  $\hspace{2.1cm}=
 \{\frac{u}{v} \ \vert \ u = r, s=g(s,1^{n-1}) \}.$\\
 Thus $F(\frac{r}{s},\frac{0}{1}^{(m-1)})=\frac{r}{s}$.  Consequently $\frac{0}{1}=0_{S^{-1}R}$.

 $(2)$ $(\Longrightarrow):$ Let $\frac{r}{s}=0_{S^{-1}R}$ for $r \in R, s \in S$. By $(1)$, we have $\frac{r}{s}=\frac{0}{1}$. Hence there exists $t \in S$ such that 
 
 $\hspace{0.6cm} 0 \in g(t,f(g(r,1^{n-1)}),-g(0,s,1^{(n-2)}),0^{(m-2)}),1^{(n-2)})$.\\
 Therefore $ 0 \in g(t,f(r,0^{(m-1)}),1^{(n-2)})$. It means $g(t,r,1^{(n-2)})=0$. \\
 $(\Longleftarrow): $ Let $g(t,r,1^{(n-2)})=0$ for some $t \in S$. Then $ 0 = g(t,f(r,0^{(m-1)}),1^{(n-2)})$. Since  $r=g(r,1^{n-1)})$ and $0=g(0,s,1^{(n-2)})$,  we get 
 
 $ 0 = g(t,f(g(r,1^{n-1)}),g(0,s,1^{(n-2)}),0^{(m-2)}),1^{(n-2)})$.\\ Then $\frac{r}{s}=\frac{0}{1}$ and so $\frac{r}{s}=0_{S^{-1}R}$, by $(1)$.

 $(3)$ Let  $s \in S$. It is clear that  $0= g(0,1^{(n-1)})$. Then we get

 $\hspace{0.5cm}0 = g(1,f(g(s,1,0^{(n-2)}),-g(1,s,1^{(n-2)}),0^{(m-2)}),1^{(n-2)}).$\\
 It means $\frac{s}{s}=\frac{1}{1}$. Now, we show that $\frac{1}{1}=1_{S^{-1}R}$.
 Let $r \in R$ and $s \in S$. Then we have

 $\hspace{0.6cm}G(\frac{r}{s},\frac{1}{1}^{(n-1)})=\frac{g(r,1^{(n-1)}}{g(s,1^{(n-1)})}=\frac{r}{s}.$\\
This implies that $\frac{1}{1}=1_{S^{-1}R}$.

$(4)$ Let $r \in R$ and $s,s^\prime \in S$. Clearly, 

$F(\frac{r}{s^\prime},{\frac{0}{s}}^{(m-1)})=\frac{f(g(r,s^{(m-1)},1^{(n-m)}),g(s^\prime,0,s^{(m-2)},1^{(n-m)})^{(m-1)})}{g(s^\prime,s,1^{(n-2)})}$

$\hspace{2.2cm}=\frac{f(g(r,s^{(m-1)},1^{(n-m)}),0^{(m-1)})}{g(s^\prime,s^{(m-1)},1^{(n-m)})}$

 $\hspace{2.2cm}=\frac{g(r,s^{(m-1)},1^{(n-m)})}{g(s^\prime,s^{(m-1)},1^{(n-m)})}$.\\
 On the other hand,

 $F(\frac{r}{s^\prime},{\frac{0}{1}}^{(m-1)})=\frac{f(g(r,1^{(n-1)}),g(s^\prime,0,1^{(n-2)})^{(m-1)})}{g(s^\prime,1^{(n-1)})}$

$\hspace{2.2cm}=\frac{f(g(r,1^{(n-1)}),0^{(m-1)})}{g(s^\prime,1^{(n-1)})}$

 $\hspace{2.2cm}=\frac{g(r,1^{(n-1)})}{g(s^\prime,1^{(n-1)})}.$
 \end{proof}
 \begin{definition}
 Let $R$ be a Krasner $(m,n)$-hyperring and $S$ be an $n$-ary multiplicative subset of $R$ with $1 \in S$. The mapping $\phi :R \longrightarrow S^{-1}R$, defined by $r \longrightarrow \frac{r}{1}$, is called natural map. 
 \end{definition}
 \begin{theorem} \label{12}
 The natural map $\phi$ is a homomorphism of Krasner $(m,n)$-hyperring.
 \end{theorem}
 \begin{proof}
 Let $R$ be a Krasner $(m,n)$-hyperring and $S$ be an $n$-ary multiplicative subset of $R$ with $1 \in S$. For all $r_1^m \in R$, we get
 
 $\hspace{1cm}\phi(f(r_1^m))=\frac{f(r_1^m)}{1}$
 
  $\hspace{2.5cm}=\frac{f(g(r_1,1^{(n-1)}),g(r_2,1^{(n-1)}),...,g(r_m,1^{(n-1)}))}{g(1^{(m)},1^{(n-m)})}$
  
 $\hspace{2.5cm}=\{\frac{r}{1}\ \vert \ r \in f(g(r_1,1^{(n-1)}),g(r_2,1^{(n-1)}),...,g(r_m,1^{(n-1)})\}$
 
 $\hspace{2.5cm}=F(\frac{r_1}{1},...,\frac{r_m}{1})$
 
 $\hspace{2.5cm}=F(\phi(r_1),...,\phi(r_m))$.\\
 Also, for all $r_1^n \in R$, we have
 
 $\hspace{1cm}\phi(g(r_1^n))=\frac{g(r_1^n)}{1}$
 
 $\hspace{2.5cm}=\frac{g(r_1^n)}{g(1^{(n)})}$

 $\hspace{2.5cm}=G(\frac{r_1}{1},...,\frac{r_1}{1})$
 
 $\hspace{2.5cm}=G(\phi(r_1),...,\phi(r_n))$.
 \end{proof}
 \begin{theorem} \label{13}
 Let $\frac{r}{s}$ be an nonzero element of $S^{-1}R$. Then
 
 $1)$ For all $s \in S$, $\phi(s)$ is an
 invertible element of $S^{-1}R$.
 
 $2)$ If $\phi(r)=0$, then there exists $t \in S$ such that $g(t,r,1^{(n-2)})=0$.
 
 $3)$ $\frac{r}{s}=G(\phi(r),\phi(s)^{-1},{\frac{1}{1}}^{(n-2)})$, for all $\frac{r}{s} \in S^{-1}R$. 
 \end{theorem}
 \begin{proof}
 $(1)$ Let $s \in S$. Then we have
 
 $\hspace*{1cm}G(\frac{s}{1},\frac{1}{s},{\frac{1}{1}}^{(n-2)})=\frac{g(s,1^{(n-1)})}{g(1,s,1^{(n-2)})}$
 
 $\hspace{3.4cm}=\frac{g(s,1^{(n-1)})}{g(s,1^{(n-1)})}$
 
 $\hspace{3.4cm}=\frac{1}{1} \hspace*{3cm} \text{by Lemma \ref{11} (3) }$
 
 $\hspace{3.4cm}=1_{S^{-1}R}.$
 
 $(2)$ It is clear by \ref{11} (2).
 
 $(3)$ Let $\frac{r}{s} \in S^{-1}R$. Then

 $\hspace*{1cm}\frac{r}{s}=\frac{g(r,1^{(n-1)})}{g(s,1^{(n-1)})}$
 
 $\hspace{1.3cm}=G(\frac{r}{1},\frac{1}{s},{\frac{1}{1}}^{(n-2)})$
 
  $\hspace{1.3cm}=G(\phi(r),\phi(s)^{-1},{\frac{1}{1}}^{(n-2)})$.
 \end{proof}
 \begin{theorem} \label{14}
 Let $(R_1,f_1,h_1)$ and $(R_2,f_2,g_2)$ be two Krasner $(m,n)$-hyperrings and $S$ be an $n$-ary multiplicative subset of $R_1$ with $1 \in S$. Let $k:R_1\longrightarrow R_2$ be a homomorphism such that for each $s \in S$, $k(s)$ is an invertible element of $R_2$. Then there exists an unique homomorphism $h:S^{-1}R_1 \longrightarrow R_2$ such that $ho\phi=k$.
 \end{theorem}
 \begin{proof}
 Let $(R_1,f_1,h_1)$,  $(R_2,f_2,g_2)$  and $(S^{-1}R_1,G,F)$ be  Krasner $(m,n)$-hyperrings such that $S$ is an $n$-ary multiplicative subset of $R_1$ and $1 \in S$. Define mapping $h$ from $S^{-1}R_1$ to $ R_2$ as follows: 
 
 $h(\frac{r}{s})=g_2(k(a),k(s)^{-1},1^{(n-2)})$.\\
 We need to show that $h$ is well defined. Let $\frac{r_1}{s_1}=\frac{r^\prime}{s^\prime}$. Then there exists $t \in S$ such that 
 
 $0 \in g_1(t,f_1(g_1(r,s^\prime,1^{(n-2)}),-g_1(r^\prime,s,1^{(n-2)}),0^{(m-2)}),1^{(n-2)})$.
 
$\hspace{0.3cm}=f_1(g_1(t,r,s^\prime,1^{(n-2)}),-g_1(t,r^\prime,s,1^{(n-2)}),0^{(m-2)})$.\\
%$\hspace{0.3cm}=f_1(g_1(t,r,s^\prime,1^{(n-2)}),-g_1(t,r^\prime,s,1^{(n-2)}),0^{(m-2)})$\\
 Hence

 $0 \in k(f_1(g_1(t,r,s^\prime,1^{(n-2)}),-g_1(t,r^\prime,s,1^{(n-2)}),0^{(m-2)})$
  
$\hspace{0.3cm}=f_2(k(g_1(t,r,s^\prime,1^{(n-2)}),
k(-g_1(t,r^\prime,s,1^{(n-2)}),k(0)^{(m-2)})$

$\hspace{0.3cm}=f_2(k(g_1(g_1(t,1^{(n-1)}),g_1(r,1^{(n-1)}),g_1(s^\prime,1^{(n-1)}),1^{(n-3)})),$

$\hspace{0.7cm}k(-g_1(g_1(t,1^{(n-1)}),g_1(r^\prime,1^{(n-1)}),g_1(s,1^{(n-1)}),1^{(n-3)})),k(0)^{(m-2)})$

$\hspace{0.3cm}=f_2(g_2(k(g_1(t,1^{(n-1)})),k(g_1(r,1^{(n-1)})),k(g_1(s^\prime,1^{(n-1)})),1^{(n-3)}),$

$\hspace{0.7cm}-g_2(k(g_1(t,1^{(n-1)})),k(g_1(r^\prime,1^{(n-1)})),k(g_1(s,1^{(n-1)})),1^{(n-3)})),k(0)^{(m-2)})$

$\hspace{0.3cm}=f_2(g_2(k(t),k(r)),k(s^\prime),1^{(n-3)}),$

$\hspace{0.7cm}-g_2(k(t),k(r^\prime),k(s),1^{(n-3)})),0^{(m-2)})$

 $\hspace{0.3cm}=f_2(g_2(k(t),g_2(k(r)),k(s^\prime),1^{(n-2)}),1^{(n-2)}),$

$\hspace{0.7cm}-g_2(k(t),g_2(k(r^\prime),k(s),1^{(n-2)}),1^{(n-2)}),0^{(m-2)})$

 $\hspace{0.3cm}=g_2(k(t),f_2(g_2(k(r)),k(s^\prime),1^{(n-2)}),$

$\hspace{0.7cm}-g_2(k(r^\prime),k(s),1^{(n-2)})),0^{(m-2)}),1^{(n-2)})$.\\
%$g_1(t,s^\prime,1^{(n-2)}) \in f_1(g_1(t,r^\prime,s,1^{(n-2)}),0^{(m-1)})$\\
Since $k(t)$, $k(s)$ and $k(s^\prime)$ are  invertible elements in $R_2$,  we get

$0=g_2(k(t)^{-1}, k(s)^{-1},k(s^\prime)^{-1},1^{(n-4)},0)$

$\hspace{0.3cm} \in g_2(g_2(k(t)^{-1}, k(s)^{-1},k(s^\prime)^{-1},1^{(n-3)}),k(t),f_2(g_2(k(r),k(s^\prime),1^{(n-2)}),$

$\hspace{0.7cm}-g_2(k(r^\prime),k(s),1^{(n-2)})),0^{(m-2)}),1^{(n-3)})$

$\hspace{0.3cm} = g_2(g_2(k(t)^{-1},k(t),1^{(n-2)}), g_2(k(s)^{-1},k(s^\prime)^{-1},1^{(n-2)}),f_2(g_2(k(r),k(s^\prime),1^{(n-2)}),$

$\hspace{0.7cm}-g_2(k(r^\prime),k(s),1^{(n-2)})),0^{(m-2)}),1^{(n-3)})$

$\hspace{0.3cm} = g_2(1, g_2(k(s)^{-1},k(s^\prime)^{-1},1^{(n-2)}),f_2(g_2(k(r),k(s^\prime),1^{(n-2)}),$

$\hspace{0.7cm}-g_2(k(r^\prime),k(s),1^{(n-2)})),0^{(m-2)}),1^{(n-3)})$

$\hspace{0.3cm} = f_2 (g_2(g_2(k(s)^{-1},k(s^\prime)^{-1},1^{(n-2)}),g_2(k(r),k(s^\prime),1^{(n-2)}),1^{(n-2)})$

$\hspace{0.7cm}-g_2(g_2(k(s)^{-1},k(s^\prime)^{-1},1^{(n-2)}),g_2(k(r^\prime),k(s),1^{(n-2)}),1^{(n-2)}),0^{(m-2)})$

$\hspace{0.3cm} = f_2 (g_2(g_2(k(s^\prime)^{-1},k(s^\prime),1^{(n-2)}),g_2(k(r),k(s)^{-1},1^{(n-2)}),1^{(n-2)})$

$\hspace{0.7cm}-g_2(g_2(k(s),k(s)^{-1},1^{(n-2)}),g_2(k(r^\prime),k(s^\prime)^{-1},1^{(n-2)}),1^{(n-2)}),0^{(m-2)})$

$\hspace{0.3cm} = f_2 (g_2(k(r),k(s)^{-1},1^{(n-2)})-g_2(k(r^\prime),k(s^\prime)^{-1},1^{(n-2)}),0^{(m-2)})$

$\hspace{0.3cm} =f_2(h(\frac{r}{s}),h(\frac{r^\prime}{s^\prime}),0^{(m-2)}).$\\
Then we coclude that $h(\frac{r}{s})=h(\frac{r^\prime}{s^\prime})$.\\
 We must show that the mapping $h$ is an homomorphism. Let $r_1^m \in R_1$ and $s_1^m \in S$. Then we get
 
 $h(F(\frac{r_1}{s_1},...,\frac{r_m}{s_m})$
 
 $=h(\frac{f_1(g_1(r_1,s_2^m,1^{(n-m)}),g_1(s_1,r_2,s_3^m,1^{(n-m)}),...,g_1(s_1^{m-1},r_m,1^{(n-m)})}{g_1(s_1^m,1^{(n-m)})})$

 $=g_2(k(f_1(g_1(r_1,s_2^m,1^{(n-m)}),...,g_1(s_1^{m-1},r_m,1^{(n-m)})),k(g_1(s_1^m,1^{(n-m)}))^{-1},1^{(n-2)})$

 $=g_2(f_2(k(g_1(r_1,s_2^m,1^{(n-m)}),...,k(g_1(s_1^{m-1},r_m,1^{(n-m)})),k(g_1(s_1^m,1^{(n-m)}))^{-1},1^{(n-2)})$
 
 $=g_2(f_2(g_2((k(r_1),k(s_2),...,k(s_m),k(1)^{(n-m)}),...,g_2(k(s_1),...,k(s_{m-1}),k(r_m),k(1)^{(n-m)})),$
 
 $\hspace{0.4cm}g_2(k(s_1)^{-1},...,k(s_m)^{-1},k(1)^{(n-m)})),1^{(n-2)})$
 
 $=f_2(g_2(g_2(k(s_1)^{-1},...,k(s_m)^{-1},k(1)^{(n-m)}),g_2(k(r_1),k(s_2),...,k(s_m),k(1)^{(n-m}),1^{(n-2)})),$

 $\hspace{0.4cm}...,$
 
 $\hspace{0.4cm}g_2(g_2(k(s_1)^{-1},...,k(s_m)^{-1},k(1)^{(n-m)}),g_2(k(s_1),...,k(s_{m-1}),k(r_m),k(1)^{(n-m)})),1^{(n-2)}))$

$=f_2(g_2(k(r_1),k(s_1)^{-1},1^{(n-2)}),...,g_2(k(r_m),k(s_m)^{-1},1^{(n-2)}))$

$=f_2(h(\frac{r_1}{s_1}),...,h(\frac{r_m}{s_m}))$.\\
Also,  we have

$h(G(\frac{r_1}{s_1},...,\frac{r_n}{s_n})=h(\frac{g_1(r_1^n)}{g_1(s_1^n)})$

$\hspace{2.3cm}=g_2(k(g_1(r_1^n)),k(g_1(s_1^n))^{-1},1^{(n-2)})$
 
$\hspace{2.3cm}=g_2(g_2(k(r_1),...,k(r_n)),g_2(k(s_1)^{-1},...,k(s_n)^{-1}),1^{(n-2)})$

 $\hspace{2.3cm}=g_2(g_2(k(r_1),k(s_1)^{-1}),1^{(n-2)}),...,g_2(k(r_n),k(s_n)^{-1}),1^{(n-2)}))$

 $\hspace{2.3cm}=g_2(h(\frac{r_1}{s_1}),...,h(\frac{r_n}{s_n}))$\\
 for $r_1^n \in R_1$ and $s_1^n \in S$. Consequently, $h$ is a homomorphism. Now, suppose that $h^\prime$ is  another homomorphism frome $S^{-1}R_1$ to $R_2$ with $h^\prime o\phi=k$.  Then  we obtain

 $h(\frac{r}{s})=h(G(\frac{r}{1},\frac{1}{s},\frac{1}{1}^{(n-2)}))$

 $\hspace{0.8cm}=g_2(h(\frac{r}{1}),h(\frac{1}{s}),h(\frac{1}{1})^{(n-2)})$

  $\hspace{0.8cm}=g_2(h(\phi(r)),h(\phi(s)^{-1}),1^{(n-2)})$
  
  $\hspace{0.8cm}=g_2(h(\phi(r)),(h(\phi(s))^{-1},1^{(n-2)})$

 $\hspace{0.8cm}=g_2(k(r),k(s)^{-1},1^{(n-2)})$
 
 $\hspace{0.8cm}=g_2(h^\prime(\phi(r)),(h^\prime(\phi(s))^{-1},1^{(n-2)})$

 $\hspace{0.8cm}=g_2(h^\prime(\frac{r}{1})),(h^\prime(\frac{s}{1}))^{-1},1^{(n-2)})$

 $\hspace{0.8cm}=g_2(h^\prime(\frac{r}{1})),h^\prime(\frac{1}{s}),1^{(n-2)})$

 $\hspace{0.8cm}=h^\prime(G(\frac{r}{1},\frac{1}{s},\frac{1}{1}^{(n-2)}))$

 $\hspace{0.8cm}=h^\prime(\frac{r}{s}).$\\
 for every $\frac{r}{s} \in S^{-1}R$. It implies that the homomorphism $h$ is unique. Thus the proof is completed.
 \end{proof}
 \begin{corollary} \label{1361}
 Let $(R_1,f_1,h_1)$ and $(R_2,f_2,g_2)$ be two Krasner $(m,n)$-hyperrings and $S$ be an $n$-ary multiplicative subset of $R_1$ with $1 \in S$. Let $k:R_1\longrightarrow R_2$ be a homomorphism such that 
 
 $i)$ $k(s)$ is an invertible element of $R_2$ for each $s \in S$. 
 
 $ii)$  $k(r_1)=0$ for  $r_1 \in R_1$ implies that $g_1(t,r_1,1^{(n-2)})=0$, for some $t \in S$.
 
 $iii)$ for each $r_2 \in R_2$, $r_2=g_2(k(r_1),k(s)^{-1},1^{(n-2)})$ where $r_1 \in R_1$ and $s \in S$.\\
 Then there exists an unique isomomorphism $h:S^{-1}R_1 \longrightarrow R_2$ such that $ho\phi=k$.
 \end{corollary}
 \begin{proof}
 By using an argument similar to that in the proof of Theorem \ref{14}, one can easily complete the proof.
 \end{proof}
  \begin{theorem} \label{1362}
 If $R$ is  an $n$-ary hyperintegral domain, then $S^{-1}R$ is an $n$-ary hyperintegral domain.
 \end{theorem}
\begin{proof}
Let $G(\frac{r_1}{s_1},...,\frac{r_1}{s_1})=0_{S^{-1}R}$ for $r_1^n \in R$ and $s_1^n \in S$. Thus $\frac{g(a_1^n)}{g(s_1^n)}=0_{S^{-1}R}$. By Lemma \ref{11} (2), we have $g(t,g(a_1^n),1^{(n-2)})=0$ for some $t \in S$. Since $R$ is  an $n$-ary hyperintegral domain and $t \neq 0$,  we have $g(a_1^n)=0$ which implies $a_1=0$ or $a_2=0$ or ... or $a_n=0$. Hence we get $\frac{a_1}{s_1}=0_{S^{-1}R}$ or $\frac{a_2}{s_2}=0_{S^{-1}R}$ or ... or $\frac{a_n}{s_n}=0_{S^{-1}R}$. Thus $S^{-1}R$ is an $n$-ary hyperintegral domain.
\end{proof} 
 \begin{theorem} \label{15}
 Let $R$ be an $n$-ary hyperintegral domain and $S=R- \{0\}$. Then  each nonzero element of  $S^{-1}R$ is invertible.
 \end{theorem}
 \begin{proof}
 Let $\frac{r}{s}$ be an nonzero element of $S^{-1}R$. Since $r \neq 0$, then $r \in S$ and so $\frac{s}{r} \in S^{-1}R$. Thus $G(\frac{r}{s},\frac{s}{r},{\frac{1}{1}}^{(n-2)})=\frac{g(r,s,1^{(n-2)})}{g(s,r,1^{(n-2)})}=\frac{1}{1}=1_{S^{-1}R}$, by Lemma \ref{11} (3).
 \end{proof}

 %%%%%%%%%%%%%%%
 %%%%%%%%%%%%%%%%%%
 \section{hyperideals in Krasner $(m,n)$-hyperring of fractions}
Let $I$ be a hyperideal of Krasner $(m,n)$-hyperring $R$ and  $S$ be an $n$-ary multiplicative subset of $R$ with $1 \in S$, then we can define that $S^{-1}I=\{\frac{a}{s}\ \vert \ a \in I, s \in S\},$ which is a hyperideal of $S^{-1}R$.
 \begin{theorem} \label{21}
 Let $R$ be a Krasner $(m,n)$-hyperring and $S$ be an $n$-ary multiplicative subset of $R$ with $1 \in S$. Let $I$ be a hyperideal of $R$. Then $I \cap S \neq \varnothing$ if and only if  $S^{-1}I=S^{-1}R$.
 \end{theorem}
 \begin{proof}
 $(\Longrightarrow):$ Let $a \in I \cap S$. Then $\frac{1}{1}=\frac{a}{a} \in S^{-1}I$. Since $I$ is  a hyperideal of $R$, we have $G(\frac{1}{1},\frac{r}{s},\frac{1}{1}^{(n-2)}) \in S^{-1}I$ for all $\frac{r}{s} \in S^{-1}R$. Since $G(\frac{1}{1},\frac{r}{s},\frac{1}{1}^{(n-2)})=\frac{g(1,r,1^{(n-2)})}{g(1,s,1^{(n-2)})}=\frac{r}{s}$, then $\frac{r}{s} \in S^{-1}I.$ Thus $S^{-1}I=S^{-1}R$.\\
 $(\Longleftarrow):$ By the homomorphism $\phi :R \longrightarrow S^{-1}R$, it implies that $\phi(1)=\frac{1}{1}$. Since $S^{-1}I=S^{-1}R$ and $\phi(1) \in S^{-1}R$, then $\phi(1) \in S^{-1}I$. Hence, there exist $a \in I$, $s \in S$ such that $\frac{1}{1}=\phi(1)=\frac{a}{s}$. So, there exists $t \in S$ such that 
 
 $0 \in g(t,f(g(a,1,1^{(n-2)}),-g(1,s,1^{(n-2)}),0^{(m-2)}),1^{(n-2)})$
 
 $\hspace{0.3cm}=g(t,f(g(a,1^{(n-1)}),-g(s,1^{(n-1)}),0^{(m-2)}),1^{(n-2)})$   
 
 $\hspace{0.3cm}=f(g(t,g(a,1^{(n-1)}),1^{(n-2)}),g(t,-g(s,1^{(n-1)}),1^{(n-2)}),0^{(m-2)})$

  $\hspace{0.3cm}=f(g(t,a,1^{(n-2)}),-g(t,s,1^{(n-2)}),0^{(m-2)})$.\\Since $g(t,a,1^{(n-2)}) \in I$, then $g(t,s,1^{(n-2)}) \in I$. Also, since $S$ is an $n$-ary multiplicative subset of $R$, then  $g(t,s,1^{(n-2)}) \in S$. Consequently, $I \cap S \neq \varnothing$.  
 \end{proof}
 If $(a,s) \in S^{-1}I$ we don$^,$t get necessarily $a \in I$, maybe $(a,s)=(a^\prime,s)$ such that $a^\prime \in I$ but $a \notin I$.
 \begin{theorem}
Let $R$ be a Krasner $(m,n)$-hyperring and $S$ be an $n$-ary multiplicative subset of $R$ with $1 \in S$. Then every hyperideal of $S^{-1}R$ is an extended hyperideal. 
 \end{theorem}
 \begin{proof}
 Suppose that $J$ is a hyperideal of $S^{-1}R$. Put $B=\{r \in R \ \vert \ \exists s \in S; \frac{r}{s} \in J\}$. Easily, it is proved that $B$ is a hyperideal of $R$. We show that $B^e=S^{-1}B=J$. Let $\frac{r}{s} \in J$. Then $r \in B$ and so $\frac{r}{s} \in S^{-1}B$ which means $J \subseteq S^{-1}B$. Now, assume that $\frac{b}{s} \in S^{-1}B$. Then there exist $b^\prime \in B$ and $s^\prime \in S$ such that $\frac{b}{s}=\frac{b^\prime}{s^\prime}$. It means there exists $t \in S$ such that
 
 $0 \in g(t,f(g(b,s^\prime,1^{(n-2)}),-g(b^\prime,s,1^{(n-2)}),0^{(m-2)}),1^{(n-2)})$
 
 $\hspace{0.3cm}=f(g(t,b,s^\prime,1^{(n-3)}),-g(t,b^\prime,s,1^{(n-3)}),0^{(m-2)}),1^{(n-2)})$.\\
 Since $g(t,b^\prime,s,1^{(n-3)}) \in B$, then $g(g(t,s^\prime,1^{(n-2}),b,1^{(n-2)})=g(t,b,s^\prime,1^{(n-3)}) \in B$. Put $t^\prime=g(t,s^\prime,1^{(n-2})$. Therefore we have $g({t^\prime}^{(m-1)},b,1^{(n-m)}) \in B$. Hence there exists $t^{\prime \prime} \in S$ such that $\frac{g({t^\prime}^{(m-1)},b,1^{(n-m)})}{t^{\prime \prime}} \in J$ and so $\frac{g({t^\prime}^{(m-1)},b,1^{(n-m)})}{g({t^{\prime \prime}}^{(m-1)},1^{(n-m+1)})} \in J$.  Then  we have
 
 $G(\frac{g({t^{\prime \prime}}^{(m-1)},1^{(n-m+1)})}{g(t^\prime,s,1^{(n-2)})},\frac{g({t^\prime}^{(m-1)},b,1^{(n-m)})}{g({t^{\prime \prime}}^{(m-1)},1^{(n-m+1)})})=\frac{g(g(t^{\prime \prime},t^\prime,1^{(n-2)})^{(m-1)},b,1^{(n-m)})}{g(g(t^{\prime \prime},t^\prime,1^{(n-2)})^{(m-1)},s,1^{(n-m)})}=\frac{b}{s} \in J$.\\
 This means $S^{-1}B \subseteq J $. Consequently,  $S^{-1}B =J $.
 \end{proof}
 Let $R$ be a Krasner (m, n)-hyperring. Then the hyperideal $M$ of $R$ is said to be
maximal if for every hyperideal $I$ of $R$, $M \subseteq  I \subseteq  R$ implies that $I = M$ or $I = R$ \cite{amer}.
 \begin{lem}
 Let $R$ be a Krasner $(m,n)$-hyperring such that $M$ is a hyperideal of $R$. If each $x \in R-M$ is invertible, then $M$ is a maximal hyperideal of $R$.
 \end{lem} 
 \begin{proof}
 The proof is similar to ordinary algebra.
 \end{proof}
 \begin{theorem}
Let $R$ be a Krasner $(m,n)$-hyperring and $P$ be an $n$-ary prime hyperideal of $R$. If $S=R-P$, then $M=\{\frac{a}{s} \ \vert \ a \in P , s \in S\}$ is the only maximal hyperideal of $S^{-1}R$.
 \end{theorem}
 \begin{proof}
 Clearly, $S=R-P$ is an $n$-ary multiplicative subset of $R$. Let $\frac{a_1}{s_1},...,\frac{a_m}{s_m} \in M$ such that $a_1^m \in P$ and $s_1^m \in S$. Then
 
 $F(\frac{a_1}{s_1},...,\frac{a_m}{s_m})=\frac{f(g(a_1,s_2^m,1^{(n-m)}),g(s_1,a_2,s_3^m,1^{(n-m)}),...,g(s_1^{m-1},a_m,1^{(n-m)}))}{g(s_1^m,1^{(n-m)})}.$\\
 Since $a_1^m \in P$, then
 
  $g(a_1,s_2^m,1^{(n-m)}),g(s_1,a_2,s_3^m,1^{(n-m)}),...,g(s_1^{m-1},a_m,1^{(n-m)})  \in P$ \\and so 
 
 $f(g(a_1,s_2^m,1^{(n-m)}),g(s_1,a_2,s_3^m,1^{(n-m)}),...,g(s_1^{m-1},a_m,1^{(n-m)})) \subseteq P$.\\
 Thus we conclude that $F(\frac{a_1}{s_1},...,\frac{a_m}{s_m}) \subseteq M$.\\
 Clearly, if $\frac{a}{r} \in M$, then $-\frac{a}{r} =\frac{-a}{r} \in M$. Also, since $0 \in P$,then  $0_{R_P}=\frac{0}{s} \in M$ for all $s \in S$. Hence $(M,F)$ is a canonical $n$-ary hypergroup.\\
 Now, let $r_1^n\in R$, $s_1^n \in S$ and $k \in\{1,...,n\}$. Then 
 
 $G(\frac{r_1}{s_1},...,\frac{r_{k-1}}{s_{k-1}},M,\frac{r_{k+1}}{s_{k+1}},...,\frac{r_n}{s_n})=\bigcup\{G(\frac{r_1}{s_1},...,\frac{r_{k-1}}{s_{k-1}},\frac{a}{s},\frac{r_{k+1}}{s_{k+1}},...,\frac{r_n}{s_n})\ \vert \ \frac{a}{s} \in M\}$

 $\hspace{4.7cm}=\bigcup \{ \frac{g(r_1^{k-1},a,r_{k-1}^n)}{g(s_1^{k-1},s,s_{k-1}^n)}\ \vert \  a \in P, s \in S\}$.\\
 Since $a \in P$, then $g(r_1^{k-1},a,r_{k-1}^n) \in P$ and so $G(\frac{r_1}{s_1},...,\frac{r_{k-1}}{s_{k-1}},M,\frac{r_{k+1}}{s_{k+1}},...,\frac{r_n}{s_n}) \subseteq M$. Thus, $(M,F,G)$ is a hyperideal of $R_P$.\\
 Suppose that $1_{R_P}=\frac{1}{1} \in M$. Then there exist $a \in P$ and $s \in S$ such that $\frac{1}{1}=\frac{a}{s}$. It implies that there exists $t \in S$ such that 
 
 $0 \in g(t,f(g(a,1,1^{(n-2)}),-g(1,s,1^{(n-2)}),0^{(m-2)}),1^{(n-2)})$
 
 $\hspace{0.3cm}=f(g(t,a,1^{(n-2)}),-g(t,s,1^{(n-2)}),0^{(m-2)})$.\\
 Since $g(t,a,1^{(n-2)}) \in P$, then $g(t,s,1^{(n-2)}) \in P$. Since $P$ is an $n$-ary prime hyperideal of $R$, then we obtain $t \in P$ or $s \in P$ which is a contradiction. Then $M$ is a proper hyperideal of $R$.\\
 Now, suppose that $\gamma \in R_P-M$.  It means $\gamma=\frac{r}{s}$ such that $r \in R-P$ and $s \in S$. Then $r \in S$ and so $\frac{s}{r} \in M$. Hence $\frac{1}{1}=G(\frac{r}{s},\frac{s}{r},\frac{1}{1}^{(n-2)}) \in M$ which is a contradiction. Consequently, $M$ is the only maximal hyperideal of $R_P$.
 \end{proof}
 \begin{theorem}
 Let $R$ be a Krasner $(m,n)$-hyperring and $S$ be an $n$-ary multiplicative subset of $R$ with $1 \in S$. If $P$ is an $n$-ary prime hyperideal of $R$ with $P \cap S=\varnothing$, then $S^{-1}P$ is an $n$-ary prime hyperideal of $S^{-1}R$.
 \end{theorem}
 \begin{proof}
 Let $G(\frac{a_1}{s_1},...,\frac{a_n}{s_n}) \in  S^{-1}P$ for $\frac{a_1}{s_1},...,\frac{a_n}{s_n} \in S^{-1}R$.
 Then we have $\frac{g(a_1^n)}{g(s_1^n)} \in  S^{-1}P$. It implies that there exists $t \in S$ such that $g(t,g(a_1^n),1^{(n-2)}) \in P$. Since $P$ is an $n$-ary prime hyperideal of $R$ and $P \cap S =\varnothing$, then $g(a_1^n) \in P$  which means there exists $1 \leq i \leq n$ with $a_i \in P$. Hence we conclude that $\frac{a_i}{s_i} \in S^{-1}P$ for some $1 \leq i \leq n$. Thus $S^{-1}P$ is an $n$-ary prime hyperideal of $S^{-1}R$.
 \end{proof}
 \begin{example} 
The set $R=\{0,1,2\}$ with the following 3-ary hyperoeration $f$ and 3-ary operation $g$ is a Krasner $(3,3)$-hyperring such that $f$ and $g$ are commutative.
\[f(0,0,0)=0, \ \ \ f(0,0,1)=1, \ \ \ f(0,1,1)=1, \ \ \ f(1,1,1)=1, \ \ \ f(1,1,2)=R,\]
\[f(0,1,2)=R, \ \ \ f(0,0,2)=2,\ \ \ f(0,2,2)=2,\ \ \ f(1,2,2)=R, \ \ \ f(2,2,2)=2,\]
$\ \ \ g(1,1,1)=1,\ \ \ \ g(1,1,2)=g(1,2,2)=g(2,2,2)=2,$\\
and for $x_1,x_2 \in R, \ g(0,x_1,x_2)=0$.

$S=\{1,2\}$ is a 3-ary multiplicative subset of Krasner $(3, 3)$-hyperring $(R, f , g)$ and hyperideal $P=\{0\}$ is a 3-ary prime hyperideal of $R$ (see example 4.10 in \cite{amer}). Thus $S^{-1}P=\{\frac{0}{1}\}$ is a $3$-ary prime hyperideal of $S^{-1}R$.
\end{example}
 Let $I$ be a hyperideal in a Krasner $(m, n)$-hyperring $R$ with
scalar identity. The radical (or nilradical) of $I$, denoted by ${\sqrt I}^{(m,n)}$
is the hyperideal $\bigcap P$, where
the intersection is taken over all $n$-ary prime hyperideals $P$ which contain $I$. If the set of all $n$-ary hyperideals containing $I$ is empty, then ${\sqrt I}^{(m,n)}$ is defined to be $R$.
Ameri and Norouzi showed that if $x \in {\sqrt I}^{(m,n)}$, then 
there exists $t \in \mathbb {N}$ such that $g(x^ {(t)} , 1_R^{(n-t)} ) \in I$ for $t \leq n$, or $g_{(l)} (x^ {(t)} ) \in I$ for $t = l(n-1) + 1$ \cite{amer}.

 \begin{lem}  \label{22}
 Let $R$ be a Krasner $(m,n)$-hyperring and $S$ be an $n$-ary multiplicative subset of $R$ with $1 \in S$. If $I$ is an $n$-ary hyperideal of $R$, then $\sqrt{S^{-1}I}^{(m,n)}=S^{-1}\sqrt{I}^{(m,n)}.$
 \end{lem}
 \begin{proof}
 Let $\frac{a}{s} \in \sqrt{S^{-1}I}^{(m,n)}$. Then there exists $k \in \mathbb{N}$ with $G(\frac{a}{s}^{(k)},\frac{1}{1}^{(n-k)}) \in S^{-1}I$ for $k \leq n$, or $G_{(l)}(\frac{a}{s}^{(k)}) \in S^{-1}I$ for $k=l(n-1)+1$. If $G(\frac{a}{s}^{(k)},\frac{1}{1}^{(n-k)}) \in S^{-1}I$, then $\frac{g(a^{(k)},1^{(n-k)})}{g(a^{(k)},1^{(n-k)})} \in S^{-1}I$. Therefore    
 $g(t,g(a^{(k)},1^{(n-k)}),1^{(n-2)}) \in I$ for some $t \in S$ and so $g(g(t,a,1^{(n-2)})^{(k)},1^{(n-k)})=g(t^{(k)},g(a^{(k)},1^{(n-k)}),1^{(n-k-1)}) \in I$. It means $g(t,a,1^{(n-2)}) \in \sqrt{I}^{(m,n)}$ and so $g(t^{(m-1)},a,1^{(n-m)}) \in \sqrt{I}^{(m,n)}$. Hence we get $\frac{g(t^{(m-1)},a,1^{(n-m)})}{g(t^{(m-1)},s,1^{(n-m)})}=\frac{a}{s} \in S^{-1}\sqrt{I}^{(m,n)}$, by Lemma \ref{11} (4). Similarly for the other case. Thus $\sqrt{S^{-1}I}^{(m,n)} \subseteq S^{-1}\sqrt{I}^{(m,n)}.$ \\Now, let
 $\frac{a}{s} \in S^{-1}\sqrt{I}^{(m,n)}$. Then we conclude $g(t,a,1^{(n-2)}) \in  \sqrt{I}^{(m,n)}$ for some $t \in S$ and so $g(t^{(m-1)},a,1^{(n-m)}) \in  \sqrt{I}^{(m,n)}$. It means that there exists $k \in \mathbb{N}$ with $g(g(t^{(m-1)},a,1^{(n-m)})^{(k)},1^{(n-k)}) \in I$ for $k \leq n$, or $g_{(l)}(g(t^{(m-1)},a,1^{(n-2)})^{(k)}) \in I$ for $k=l(n-1)+1$. If $g(g(t^{(m-1)},a,1^{(n-2)})^{(k)},1^{(n-k)}) \in I$, then we have 
 
 $G(\frac{a}{s}^{(k)},\frac{1}{1}^{(n-k)})=G(\frac{g(t^{(m-1)},a,1^{(n-2)})}{g(t^{(m-1)},s,1^{(n-2)})}^{(k)},\frac{1}{1}^{(n-k)})$

 $\hspace{2.4cm}=\frac{g(g(t^{(m-1)},a,1^{(n-2)})^{(k)},1^{(n-k)})}{g(g(t^{(m-1)},s,1^{(n-2)})^{(k)},1^{(n-k)})} \in S^{-1}I$.\\ Therefore we get $\frac{a}{s} \in \sqrt{S^{-1}I}^{(m,n)}$. Similarly for the other case. Thus $S^{-1}\sqrt{I}^{(m,n)} \subseteq \sqrt{S^{-1}I}^{(m,n)}$. Consequently, $\sqrt{S^{-1}I}^{(m,n)}=S^{-1}\sqrt{I}^{(m,n)}$.
 \end{proof}

 A hyperideal $Q \neq R$ in a Krasner $(m, n)$-hyperring $(R, f , g)$ with the
scalar identity $1_R$ is said to be $n$-ary primary if $g(x^n _1) \in Q$ and $x_i \notin Q$ implies that $g(x_1^{i-1}, 1_R, x_{ i+1}^n) \in {\sqrt Q}^{(m,n)}$ \cite{amer}.
 
 \begin{theorem}
 Let $R$ be a Krasner $(m,n)$-hyperring and $S$ be an $n$-ary multiplicative subset of $R$ with $1 \in S$. If $P$ is an $n$-ary  primary hyperideal of $R$ with $P \cap S=\varnothing$, then $S^{-1}P$ is an $n$-ary  primary hyperideal of $S^{-1}R$.
 \end{theorem}
 \begin{proof}
 Let $\frac{a_1}{s_1},...,\frac{a_n}{s_n} \in S^{-1}R$ such that $G(\frac{a_1}{s_1},...,\frac{a_n}{s_n}) \in  S^{-1}P$ .
 Then we have $\frac{g(a_1^n)}{g(s_1^n)} \in  S^{-1}P$. It implies that there exists $t \in S$ such that $g(t,g(a_1^n),1^{(n-2)}) \in P$. Since $P$ is an $n$-ary  primary  hyperideal of $R$, then  there exist $1 \leq i  \leq n$ such that at least one of the  cases hold: $a_i \in P$, $g(a_1^{i-1},1,a_{i+1}^n) \in \sqrt{P}^{(m,n)}$, $t \in \sqrt{P}^{(m,n)}$ or $g(a_1^n) \in \sqrt{P}^{(m,n)}$. If $a_i \in P$, then $\frac{a_i}{s_i} \in S^{-1}P$ and we are done. 
 %If $t \in P$ then for all $1 \leq i \leq n$, $g(t,a_i,1^{(n-2)}) \in P$ and so $g(t^{(m-1)},a_i,1^{(n-m)}) \in P$. Therefore $\frac{a_i}{s_i}=\frac{g(t^{(m-1)},a_i,1^{(n-m)})}{g(t^{(m-1)},s_i,1^{(n-m)})} \in S^{-1}P$. 
 If $g(a_1^{i-1},1,a_{i+1}^n) \in \sqrt{P}^{(m,n)}$, then $G(\frac{a_1}{s_1},...,\frac{a_{i-1}}{s_{i-1}},\frac{1}{1},\frac{a_{i+1}}{s_{i+1}},...,\frac{a_n}{s_n})=\frac{g(a_1^{i-1},1,a_{i+1}^n)}{g(s_1^{i-1},1,s_{i+1}^n)} \in S^{-1}\sqrt{P}^{(m,n)}=\sqrt{S^{-1}P}^{(m,n)}$, by Lemma \ref{22}. If $t \in \sqrt{P}^{(m,n)}$, then $g(t^{(m-1)},a_k,1^{(m-n)}) \in \sqrt{P}^{(m,n)}$, for all $1 \leq k \leq n$. Therefore $\frac{g(t^{(m-1)},a_k,1^{(m-n)})}{g(t^{(m-1)},s_k,1^{(m-n)})} \in S^{-1}\sqrt{P}^{(m,n)}=\sqrt{S^{-1}P}^{(m,n)}$ and so $\frac{a_k}{s_k} \in \sqrt{S^{-1}P}^{(m,n)}$. Therefore for each $i \neq k$, $G(\frac{a_1}{s_1},...,\frac{a_{i-1}}{s_{i-1}},\frac{1}{1},\frac{a_{i+1}}{s_{i+1}},...,\frac{a_n}{s_n}) \in \sqrt{S^{-1}P}^{(m,n)}$. Let $g(a_1^n) \in \sqrt{P}^{(m,n)}$. Theorem 4.28. in \cite{amer} shows that   $\sqrt{P}^{(m,n)}$ is an $n$-ary prime hyperideal of $R$. Hence there exists $1 \leq k \leq n $ such that $a_i \in \sqrt{P}^{(m,n)}$. It implies that $\frac{a_k}{s_k} \in S^{-1}\sqrt{P}^{(m,n)}=\sqrt{S^{-1}P}^{(m,n)}$. Therefore for each $i \neq k$, $G(\frac{a_1}{s_1},...,\frac{a_{i-1}}{s_{i-1}},\frac{1}{1},...,\frac{a_{i+1}}{s_{i+1}},\frac{a_n}{s_n}) \in \sqrt{S^{-1}P}^{(m,n)}$. Thus $S^{-1}P$ is an $n$-ary  primary hyperideal of $S^{-1}R$.
 \end{proof}
 
 A  proper hyperideal $I$ of a Krasner $(m,n)$-hyperring $(R,f,g)$ with the scalar identity $1_R$ is said to be $n$-ary 2-absorbing if for $x_1^n \in R$, $g(x_1^n) \in I$ implies that $g(x_i,x_j,1_R^{(n-2)}) \in I$ for some $1 \leq i < j \leq n$ \cite{mah}.
 
 \begin{theorem}
 Let $R$ be a Krasner $(m,n)$-hyperring and $S$ be an $n$-ary multiplicative subset of $R$ with $1 \in S$. If $P$ is an $n$-ary 2-absorbing hyperideal of $R$ with $P \cap S=\varnothing$, then $S^{-1}P$ is an $n$-ary 2-absorbing hyperideal of $S^{-1}R$.
 \end{theorem}
 \begin{proof}
 Let $G(\frac{a_1}{s_1},...,\frac{a_n}{s_n}) \in  S^{-1}P$, for $\frac{a_1}{s_1},...,\frac{a_n}{s_n} \in S^{-1}R$.
 Then we have $\frac{g(a_1^n)}{g(s_1^n)} \in  S^{-1}P$. It implies that there exists $t \in S$ such that $g(t,g(a_1^n),1^{(n-2)}) \in P$. Since $P$ is an $n$-ary 2-absorbing hyperideal of $R$, then  there exist $1 \leq i < j \leq n$ such that $g(t,a_i,1^{(n-2)}) \in P$ or  $g(a_i,a_j,1^{(n-2)}) \in P$. Hence we conclude that $\frac{a_i}{s_i} \in S^{-1}P$ for some $1 \leq i \leq n$. Thus $S^{-1}P$ is an $n$-ary prime hyperideal of $S^{-1}R$. If for some $1 \leq i \leq n$, $g(t,a_i,1^{(n-2)}) \in P$, then $g(t^{(m-1)},a_i,1^{(n-m)}) \in P$ and so $\frac{g(t^{(m-1)},a_i,1^{(n-m)})}{g(t^{(m-1)},s_i,1^{(n-m)})} \in S^{-1}P$. Hence $\frac{a_i}{s_i} \in S^{-1}P$, by Lemma \ref{11} (4). Therefore for every $1 \leq j \leq n$, $G(\frac{a_i}{s_i},\frac{a_j}{s_j},\frac{1}{1}^{(n-2)}) \in S^{-1}P$ and we are done. If $g(a_i,a_j,1^{(n-2)}) \in P$, for some $1 \leq i < j \leq n$, then $\frac{g(a_i,a_j,1^{(n-2)})}{g(s_i,s_j,1^{(n-2)})} \in S^{-1}P$ which means $G(\frac{a_i}{s_i},\frac{a_j}{s_j},\frac{1}{1}^{(n-1)}) \in S^{-1}P$. Consequently, $S^{-1}P$ is an $n$-ary 2-absorbing hyperideal of $S^{-1}R$. 
 \end{proof}
 %%%%%%%%%%%%%%%%%%%%%
 %%%%%%%%%%%%%%%%%%%%
 \section{qutient Krasner $(m,n)$-hyperring of fractions}
 Let $R$ be a Krasner $(m,n)$-hyperring and $I$ be a hyperideal of $R$. Then we consider the
set $R/I$ as follows:
 
 $R/I=\{f(r,I,0^{(m-2)}) \ \vert \ r \in R\}$.  
 \begin{lem}
 Let $R$ be a Krasner $(m,n)$-hyperring and $S$ be an $n$-ary multiplicative subset of $R$ with $1 \in S$. Let $I$ be a hyperideal of $R$ such that $S \cap I=\varnothing$. Then $\bar{S}=\{f(s,I,0^{(m-2)}) \ \vert \ s \in S\}$ is an $n$-ary multiplicative subset of $R/I$.
 \end{lem}
 \begin{proof}
 Let $f(s_1,I,0^{(m-2)}),...,f(s_n,I,0^{(m-2)}) \in \bar{S}$, for $s_1^n \in S$. Then we have 
 
 $g(f(s_1,I,0^{(m-2)}),...,f(s_n,I,0^{(m-2)}))=
 f(g(s_1^n),I,0^{m-2)})$.\\
 Since $S$ is an $n$-ary multiplicative subset of $R$, then $g(s_1^n) \in S$. It implies that $g(f(s_1,I,0^{(m-2)}),...,f(s_n,I,0^{(m-2)})) \in \bar{S}$. 
 \end{proof}
 \begin{theorem} \label{1363}
 Let $R$ be a Krasner $(m,n)$-hyperring and $S$ be an $n$-ary multiplicative subset of $R$ with $1 \in S$. Let $I$ be a hyperideal of $R$ such that $S \cap I=\varnothing$. If $\bar{S}=\{f(s,I,0^{(m-2)}) \ \vert \ s \in S\}$, then  $ \bar{S}^{-1}(R/I) \cong S^{-1}R/S^{-1}I$.
 % such that $h(f(r,I,0^{(m-2)})/f(s,I,0^{(m-2)}))=F(\frac{r}{s},S^{-1}I,0^{(m-2)})$, for $r \in R$ and $s \in S$.
 \end{theorem}
 \begin{proof}
 Define  mapping  $k:R/I \longrightarrow S^{-1}R/S^{-1}I$ as following:

 $k(f(r,I,0^{(m-2)})=F(\frac{r}{1},S^{-1}I,0_{S^{-1}R}^{(m-2)})$.\\
 It is easy to see the mapping is a homomorphism. Let $f(s,I,0^{(m-2)}) \in \bar{S}$.  Then $k(f(s,I,0^{(m-2)}))=F(\frac{s}{1},S^{-1}I,0_{S^{-1}R}^{(m-2)}).$ Since  $F(\frac{1}{s},S^{-1}I,0_{S^{-1}R}^{(m-2)}) \in S^{-1}R/S^{-1}I$, then we obtain
 
  $G(F(\frac{s}{1},S^{-1}I,0_{S^{-1}R}^{(m-2)}), F(\frac{1}{s},S^{-1}I,0_{S^{-1}R}^{(m-2)}),F(\frac{1}{1},S^{-1}I,0_{S^{-1}R}^{(m-2)})^{(n-2)}))$

  $\hspace{1cm} =F(G(\frac{s}{1},\frac{1}{s},\frac{1}{1}^{(n-2)}),S^{-1}I,0_{S^{-1}R}^{(m-2)})$

  $\hspace{1cm} =F(\frac{1}{1},S^{-1}I,0_{S^{-1}R}^{(m-2)})).$\\
Assume that $k(f(r,I,0^{(m-2)})=S^{-1}I.$ Then we have $F(\frac{r}{1},S^{-1}I,0_{S^{-1}R}^{(m-2)})=S^{-1}I$. It means $\frac{r}{1} \in S^{-1}I$. Then there exists $t \in S$ such that $g(t,r,1^{(n-2)}) \in I$. Clearly, $f(t,I,0^{(m-2)}) \in \bar{S}$ and we have
  
  $g(f(t,I,0^{(m-2)},f(r,I,0^{(m-2)}),f(1,I,0^{(m-2)})^{(n-2)})$
  
  $\hspace{1cm}=f(g(t,r,1^{(n-2)}),I,0^{(m-2)})=I$.\\Now, suppose that $F(\frac{r}{s},S^{-1}I,0_{S^{-1}R}) \in S^{-1}R/S^{-1}I$. Thus we have

  $F(\frac{r}{s},S^{-1}I,0_{S^{-1}R})=G(F  (\frac{r}{1},S^{-1}I,0_{S^{-1}R}),F  (\frac{1}{s},S^{-1}I,0_{S^{-1}R}),F  (\frac{1}{1},S^{-1}I,0_{S^{-1}R})^{(n-2)})$

  $\hspace{1cm}=G(k(f(r,I,0^{(m-2)}),
  k(f(r,I,0^{(m-2)}),F  (\frac{1}{1},S^{-1}I,0_{S^{-1}R})^{(n-2)})$.\\
  Hence, there exists an isomorphism from $ \bar{S}^{-1}(R/I) $ to  $S^{-1}R/S^{-1}I$, by Corollary \ref{1361}. It means $ \bar{S}^{-1}(R/I) \cong S^{-1}R/S^{-1}I$.
 \end{proof}

 Let $P$ be an $n$-ary prime hyperideal of Krasner $(m,n)$-hyperring $R$. Put $S=R-P$. Then $S$ is an $n$-ary multiplicative subset of $R$ such that $1 \in S$ and $0 \notin S$.  In this case, we denote $S^{-1}R=R_P$. Moreover, If $S^{-1}I$ is a hyperideal of $R_P$, then it is denoted by $IR_P$.  

 \begin{example}
 Let $R$ be a Krasner $(m,n)$-hyperring such that $P$ is  an $n$-ary prime hyperideal of $R$. Put $S=R-P$. Then
 
 $\bar{S}=\{f(s,P,0^{(n-2)}) \ \vert s \in S\}= R/P-\{f(P,0^{(n-1)})\}$\\
 is an $n$-ary multiplicative subset of $R/P$. By Theorem 4.6 in \cite{amer}, $R/P$ is an n-ary hyperintegral domain. Theorem \ref{1362} and \ref{15} show that $\bar{S}^{-1}(R/P)$ is an n-ary hyperintegral domain and  each nonzero element of  $\bar{S}^{-1}(R/P)$  is invertible. Moreover, we have

 $\bar{S}^{-1}(R/P) \cong \frac{S^{-1}R}{S^{-1}P}=\frac{R_p}{PR_p}$,\\
 by Theorem \ref{1363}.
 \end{example}
 \begin{example}
 Let $R$ be a Krasner $(m,n)$-hyperring such that $P$ and $Q$ are two $n$-ary prime hyperideals of $R$ such that $Q \subseteq P$. Put $S=R-P$. Then
 
 $\bar{S}=\{f(s,Q,0^{(n-2)}) \ \vert s \in S\}= R/Q-R/P.$\\ It is clear that $P/Q$ is an $n$-ary prime hyperideal of $R/Q$. Therefore $\bar{S}^{-1}(R/Q)=(R/Q)_{P/Q}$. By Theorem \ref{1363}, we get $(R/Q)_{P/Q} \cong \frac{R_P}{QR_P}$.
 \end{example}
%%%%%%%%%%%%%%%%%%%%%%%%%%%%%%%%%%%%%%%%%%%
%%%%%%%%%%%%%%%%%%%%%%%%%

\end{document}